\documentclass{amsart}[12pt]
\usepackage{amssymb,amsmath}
\usepackage{enumerate}
\usepackage[english]{babel}
\usepackage{color}

\newcommand{\Om} {\Omega}

\newcommand {\ii} {\infty}

\newcommand {\al} {\alpha}

\newcommand {\lb} {\lambda}

\newcommand {\sm} {\setminus}
\newcommand {\su} {\subset}

\newcommand {\mc} {\mathcal}
\newcommand {\cal} {\mathcal}
\newcommand {\supp} {\mathop{\rm supp}}

\newtheorem{teo}{Theorem}[section]
\newtheorem{pro}{Proposition}[section]
\newtheorem{cor}{Corollary}[section]

\theoremstyle{definition}

\newtheorem{df}{Definition}[section]

\title{validity space \\ of Dunford-Schwartz pointwise ergodic theorem}
\begin{document}
\date{May 7, 2017}

\begin{abstract}
We show that if a $\sigma-$finite infinite measure space $(\Om,\mu)$ is quasi-non-atomic,
then the Dunford-Schwartz pointwise ergodic theorem
holds for $f\in \cal L^1(\Om)+\cal L^\ii(\Om)$  if and {\it only if}
$\mu\{f\ge \lb\}<\ii$ for all $\lb>0$.
\end{abstract}

\author{VLADIMIR CHILIN}
\address{National University of Uzbekistan, Tashkent, Uzbekistan}
\email{vladimirchil@gmail.com; chilin@ucd.uz}
\author{SEMYON LITVINOV}
\address{Pennsylvania State University \\ 76 University Drive \\ Hazleton, PA 18202, USA}
\email{snl2@psu.edu}

\maketitle

\section{Introduction}
Let $(\Om, \mathcal A, \mu)$ be a complete measure space.
Denote  by $\cal L^0 = \cal L^0(\Om)$  the algebra of equivalence classes of almost everywhere (a.e.)
finite real-valued measurable functions on $\Om$. Let $\mc L^p\su \mc L^0$, $1\leq p\leq \ii$, be the $L^p-$space
equipped with the standard norm $\| \cdot \|_p$.

Let $T:\cal L^1+\cal L^\ii\to \cal L^1+\cal L^\ii$ be a Dunford-Schwartz operator (see Definition \ref{d1}).
Dunford-Schwartz theorem on a.e. convergence of the ergodic  averages
\begin{equation}\label{e1}
A_n(f)= A_n(T,f)=\frac 1 n \sum_{k=0}^{n-1}T^k(f)
\end{equation}
for a Dunford-Schwartz operator $T$ acting in the $L^p-$space, $1\leq p<\ii$, of real valued functions of an
arbitrary measure space was established in \cite{DS}; see also \cite[Theorem VIII.6.6]{ds}.

If $\mu(\Om)=\ii$, then there is the problem of describing the largest subspace of $\cal L^1+\cal L^\ii$
for which Dunford-Schwartz theorem is valid.
It can be noticed that the averages $A_n(f)$ converge a.e. for all $f\in \cal R_\mu$, that is,
when $f\in \cal L^1+\cal L^\ii$ is such that $\mu\{f\ge \lb\}<\ii$ for all $\lb>0$ (see, for example, \cite{ccl}).
We show that  if a $\sigma-$finite infinite measure space $(\Om,\mu)$ has finitely many atoms or its atoms
have equal measures, then
$\cal R_\mu$ is the largest subspace of $\cal L^1+\cal L^\ii$ for which
the convergence takes place: if $f\in (\cal L^1+\cal L^\ii)\sm \cal R_\mu$, then there exists a Dunford-Schwartz operator
$T$ such that the sequence $\{A_n(T,f)\}$ does not converge a.e. (Theorem \ref{t10}).

As a corollary, we derive that a fully symmetric space
$E\su \cal L^1+\cal L^\ii$  of a $\sigma-$finite quasi-non-atomic measure space (see Definition \ref{d4}) possesses
the individual ergodic theorem property
(see Definition \ref{d3}) if and only if the characteristic function of $\Om$ does not belong to $E$. In conclusion, we
outline some (classes of) fully symmetric spaces that do or do not possess the individual ergodic theorem property.

\section{Preliminaries}
 If $f \in \cal L^1 + \cal L^{\ii}$, then
a {\it non-increasing rearrangement} of $f$ is defined as
$$
\mu_t(f)=\inf \{\lb>0: \ \mu\{|f| > \lb\} \leq t\}, \ \  t>0
$$
(see \cite[Ch.II, \S 2]{kps}).

A Banach space $(E, \| \cdot \|_E)\su \cal L^1 + \cal L^{\ii}$ is called {\it symmetric (fully symmetric)} if
$$
f \in E, \ g \in \cal L^1 + \cal L^{\ii}, \ \mu_t(g)\leq \mu_t(f) \ \ \forall \ t>0
$$
(respectively,
$$
f \in E, \ g \in \cal L^1 + \cal L^{\ii}, \ \int \limits_0^s\mu_t(g)dt\leq  \int \limits_0^s\mu_t(f)dt \ \ \forall \ s>0
 \ (\text {writing } \  g \prec\prec f))
$$
implies that $g \in E$ and $\| g\|_E\leq \| f\|_E$.

Immediate examples of fully symmetric spaces are $\cal L^1\cap \cal L^{\infty}$  with the norm
$$
\|f\|_{\cal L^1\cap \cal L^{\infty}}=\max \left \{ \|f\|_1, \|f\|_{\infty} \right\}
$$
and $\cal L^1 + \cal L^{\infty}$ with the norm
$$
\|f\|_{\cal L^1 + \cal L^{\infty}}=\inf \left \{ \|g\|_1+ \|h\|_{\infty}: \ f = g + h, \ g \in \cal L^1, \ h \in \cal L^{\infty} \right \}=
\int_0^1 \mu_t(f) dt
$$
(see \cite[Ch. II, \S 4]{kps}).

Define
$$
\cal R_\mu= \{f \in \cal L^1 + \cal L^{\infty}: \ \mu_t(f) \to 0 \text{ \ as \ } t\to \ii\}.
$$
It is clear that $\mc R_\mu$ admits a more direct description:
$$
\cal R_\mu= \{f \in \cal L^1 + \cal L^{\infty}:  \mu\{f> \lb\}<\ii \text{ \ for all \ } \lb>0\}.
$$
Note that if $\mu(\Om) < \infty$, then $\cal R_\mu$ is simply $\cal L^1$.
Therefore, we will be concerned with infinite measure spaces.

By \cite[Ch.II, \S 4, Lemma 4.4]{kps}, $(\cal R_\mu, \|\cdot\|_{\cal L^1 + \cal L^\ii})$ is a symmetric space. In addition,
$\cal R_\mu$ is the closure of $\cal L^1\cap \cal L^{\infty}$  in $\cal L^1 + \cal L^{\infty}$ (see \cite[Ch.II, \S 3, Section 1]{kps}).
Furthermore, it follows from definitions of $\cal R_\mu$ and $\| \cdot \|_{\cal L^1 + \cal L^{\infty}}$ that if
$$
f \in \cal R_\mu, \ g \in \cal L^1+\cal L^{\infty}, \text{\ and \ } g \prec\prec f,
$$
then $g\in \cal R_\mu$ and $\| g\|_{\cal L^1 + \cal L^{\infty}} \leq \| f\|_{\cal L^1 + \cal L^{\infty}}$. Therefore
$(\cal R_\mu, \|\cdot\|_{\cal L^1 + \cal L^\ii})$ is a fully symmetric space.

\begin{df}\label{d1}
A linear operator $T: \cal L^1 + \cal L^{\ii} \to  \cal L^1 + \cal L^{\ii}$ is called a {\it Dunford-Schwartz operator} if
$$
\| T(f)\|_1\leq \| f\|_1 \ \ \forall \ \ f\in \cal L^1 \text{ \ \ and \ \ } \| T(f)\|_{\ii}\leq \| f\|_\ii \ \ \forall \ f \in \cal L^{\ii}.
$$
\end{df}

In what follows, we will write $T\in DS$ to indicate that $T$ is a Dunford-Schwartz operator. If $T$ is, in addition,
positive, we shall write $T\in DS^+$. It is clear that
$$
\|T\|_{\cal L^1 + \cal L^{\ii} \to  \cal L^1 + \cal L^{\ii}} \leq 1
$$
for all $T\in DS$ and $T(f )\prec\prec f$ for all $f \in \cal L^1 + \cal L^{\ii}$ \cite[Ch.II, \S 3, Sec.4]{kps}.
Therefore $T(E) \subset E$ for every fully symmetric space
$E$ and
$$
\| T\|_{E \to  E} \leq 1
$$
(see \cite[Ch.II, \S 4, Sec.2]{kps}).
In particular,  $T(\cal R_\mu) \subset \cal R_\mu$, and the restriction of $T$ on $\cal  R_\mu$ is a linear contraction.

\begin{df} \label{d3}
We say that a fully symmetric space $E\su \mc L^1(\Om)+\mc L^\ii(\Om)$ possesses the
{\it individual ergodic theorem property}, writing $E\in$ IET$\,(\Om)$, if for every $f\in E$ and $T\in DS$ the averages (\ref{e1}) converge a.e. to some $\widehat f\in E$.
\end{df}

Let $\chi_E$ be the characteristic function of a set $E \in \mathcal A$. Denote
$\mathbf 1 = \chi_\Om$. The following fact was noticed in \cite[Proposition 2.1]{ccl}.

\begin{pro}\label{p1}
If $\mu(\Om) = \infty$, then a symmetric space $E \su \cal L^1+\cal L^{\infty}$ is contained in
$\cal R_\mu$ if and only if $\mathbf 1\notin E$.
\end{pro}

Here is a Dunford-Schwartz pointwise ergodic theorem in a fully symmetric space $E\su \cal R_\mu$ \cite[Theorem 4.3]{ccl}:

\begin{teo}\label{t1} Let $(\Om, \mu)$  be an infinite  measure space. If  $E\su \mc L^1+\mc L^\ii$
is a fully symmetric space with $\mathbf 1\notin E$, then $E\in IET(\Om)$. In particular, $\mc R_\mu\in IET(\Om)$.
\end{teo}

\begin{df}\label{d4}
 We say that a measure space $(\Om,\mu)$  is {\it quasi-non-atomic}
if it has finitely many atoms or its atoms have the same measure.
\end{df}

Denote
$$
\mc C_\mu=\left \{ f\in \mc L^1+\mc L^\ii: \ \{A_n(T,f)\} \text{ \ converges a.e. for all} \ T\in DS\right \}.
$$
By Theorem \ref{t1}, $\mc R_\mu\su\mc C_\mu$. 

Our main result is Theorem \ref{t10} below stating that  if a $\sigma-$finite measure space $(\Om,\mu)$ 
is quasi-non-atomic, then
the fully symmetric space $\mc R_\mu$ is maximal relative to Dunford-Schwartz pointwise ergodic theorem, that is,
$$
\mc C_\mu=\mc R_\mu.
$$

We will begin with the case when $\Om=(0,\ii)$ equipped with Lebesgue measure; see Theorem \ref{t2}.

In order to establish Theorem \ref{t2} for a quasi-non-atomic measure space, we will need some properties of non-atomic
$\sigma-$finite measure spaces.

Let  $(\Om_1, \mathcal A_1, \mu_1)$ and $(\Om_2, \mathcal A_2, \mu_2)$ be measure spaces.
A mapping $\sigma: \Om_1\to \Om_2$ is said to be a {\it measure-preserving transformation (m.p.t.)} if
$$
\sigma^{-1}(E) \in  \mc A_1 \text{ \ and \ } \mu_1\left (\sigma^{-1}(E)\right )=\mu_2(E) \text{ \ for every \ } E \in \mc A_2.
$$

If $f \in \cal L^0(\Om)$, denote $\supp(f) = \{f \neq 0\}$. The following
property of  non-atomic measure spaces can be found in \cite[Ch.2, Corollary 7.6]{bs}.

\begin{teo}\label{t7}
Let $(\Om, \mu)$ be a non-atomic $\sigma-$finite measure space. Then, given $0\leq f\in\mc R_\mu$,
there is a surjective m.p.t. $\sigma: \supp(f)\to\supp(\mu_t(f))$
such that
$$
f =  \mu_t(f) \circ \sigma
$$
on $\supp(f)$.
\end{teo}

Let $\nabla$ be a complete Boolean algebra. Denote by $Q(\nabla)$ the Stone compact of $\nabla$, and let $C_\infty(Q(\nabla))$ be the algebra of continuous functions $f: Q(\nabla)\rightarrow [-\infty,+\infty]$ assuming possibly the
values $\pm\infty$  on nowhere dense subsets of $Q(\nabla)$ \cite[1.4.2]{ku}.  Let $C(Q(\nabla))$ be the Banach algebra of continuous functions $f: Q(\nabla)\rightarrow \mathbb R$ relative to the norm
$$
\| f \|_{\infty} = \max_{t \in Q(\nabla)}|f(t)|.
$$

We will employ the following extension result.

\begin{teo}\label{t4}
If $\varphi:\nabla_1\to \nabla_2$ is an isomorphism of complete Boolean algebras,
then there exists a unique  isomorphism $\Phi: C_\infty(Q(\nabla_1))\rightarrow C_\infty(Q(\nabla_2))$
such that $\Phi(e)=\varphi(e)$ for all $e \in  \nabla_1$.
\end{teo}

\begin{proof}
Show first that the isomorphism $\varphi$ extends up to an isomorphism $\Phi_0: C(Q(\nabla_1))\rightarrow C(Q(\nabla_2))$.
Let $\mathbb R(\nabla_1) $ be the  subalgebra of step-elements of the algebra
$C_\infty(Q(\nabla_1))$, that is, elements of the form
$f=\sum\limits_{i=1}^k \lambda_i e_i$, where $e_i \in \nabla_1$, $e_i e_j = 0$ if $i \neq j$,
and $\lambda_i \in \mathbb R$, $i=1,..., k$.
It is clear that the  subalgebra $\mathbb R(\nabla_1) $ is  dense  in the Banach algebra $ (C(Q(\nabla_1)), \| \cdot \|_{\infty})$.

Put
$$ \Phi_0(f) = \sum\limits_{i=1}^k \lambda_i \varphi(e_i).$$
Since $\varphi$  is an isomorphism of Boolean algebras, it follows that $\Phi_0: \mathbb R(\nabla_{1})\to \mathbb R(\nabla_{2})$
is an isomorphism and an isometry with respect to the norm $\| \cdot \|_{\infty}$.  Due to the density of the subalgebra
$\mathbb R(\nabla_{i}) $ in the Banach algebra $ (C(Q(\nabla_i)), \| \cdot \|_{\infty})$, $i=1,2$,
$\Phi_0$ uniquely extends to an  isomorphism
$\Phi_0: C(Q(\nabla_1))\to C(Q(\nabla_2))$. It is clear that $\Phi_0(e)=\varphi(e)$ for all $e \in  \nabla_{1}$.

If $f \in C_\infty(Q(\nabla_1))$, then there exists a  partition
$\{e_i \} $ of unity $\mathbf 1_{\nabla_{1}} $, such that $ f e_i \in
C(Q(\nabla_1))$ for each $i$. Since $\varphi $ is an  isomorphism of Boolean algebras,
it follows that $ \{\varphi (e_i) \} $ is a partition of unity $\mathbf 1_{\nabla_{2}} $. Therefore, there is a unique
$g \in C_\infty(Q(\nabla_2))$, such that
$$
\Phi_0 (fe_i) \varphi (e_i) = g \varphi (e_i) \text{ \ for all \ }  i.
$$
If $ \{q_j \} $ is another partition of unity $\mathbf 1_{\nabla_{1}}$ such that  $fq_j \in C(Q(\nabla_1))$ for all $j$, then
$\{ p_{ij} = e_iq_j\}$  is also a  partition of unity $\mathbf 1_{\nabla_{1}} $ and $f p_{ij} \in C(Q(\nabla_1))$ for any $ i, j $.
If $h\in C_\infty(Q(\nabla_2))$ is such that $h \varphi(q_j) = \Phi_0(fq_j) \varphi(q_j) $
for all $ j $, then
$$
g \varphi(p_{ij}) = g \varphi(e_i)\varphi(q_j) =
\Phi_0(fe_i) \varphi(e_i) \varphi(q_j) = \Phi_0(fe_iq_j) =
$$
$$
= \Phi_0(fq_j) \varphi(q_j) \varphi(e_i) = h \varphi(q_j) \varphi(e_i) = h \varphi (p_{ij})
$$
for any $ i, j $. Since
$$
 \sup\limits_{i, j } \varphi(p_{ij}) = \left (\sup\limits_{i} \varphi(e_i)\right )
\left (\sup\limits_ {j} \varphi(q_j)\right ) = \mathbf 1_{\nabla_{2}},
$$
we have $h=g$. Therefore, we can define the mapping $\Phi: C_\infty(Q(\nabla_1))\rightarrow C_\infty(Q(\nabla_2))$
by $\Phi(f) = g $. Clearly, $\Phi: C_\infty(Q(\nabla_1))\rightarrow C_\infty(Q(\nabla_2))$ is a unique isomorphism
such that $\Phi(e) = \varphi(e) $ for each $e \in \nabla_1$.
\end{proof}

If $\nabla_\mu$ is the complete Boolean algebra of the classes $e=[E]$ of $\mu-$a.e. equal sets in
$\mc A$, then
$$
\widehat{\mu}(e) = \mu(E)
$$
is a strictly positive  measure on $\nabla_\mu$.
The following is a refinement of Theorem \ref{t4} in the case of measure spaces.

\begin{teo}\label{t5}
Let $(\Om_i, \mu_i)$ be a $\sigma-$finite measure space, $ i=1,2$,
and let $\varphi:\nabla_{\mu_1}\to \nabla_{\mu_2}$ be an isomorphism such that
\begin{equation}\label{e2}
\widehat{\mu}_2(\varphi(e)) = \widehat{\mu}_1(e), \ e\in  \nabla_{\mu_1}.
\end{equation}
Then there exists a unique  isomorphism $\Phi: \cal L^0(\Om_1)\rightarrow \cal L^0(\Om_2)$ such that
\begin{enumerate}
\item $\Phi(e)=\varphi(e)$ for all $e \in \nabla_{\mu_1}$;
\item $\Phi: \cal L^1(\Om_1) \to \cal L^1(\Om_2)$ and $\Phi: \cal L^\ii(\Om_1) \to \cal L^\ii(\Om_2)$
are bijective linear isometries;
\item If $f_n, f \in \cal L^0(\Om_1)$ and $g_n, g \in \cal L^0(\Om_2)$, then $\Phi(f_n)\rightarrow\Phi(f)$ ($\mu_2-$a.e.)
if and only if $f_n \to f$ ($\mu_1-$a.e.)
and $\Phi^{-1}(g_n)\to \Phi^{-1}(g)$ ($\mu_1-$a.e.)
if and only if $g_n \to g$ ($\mu_2-$a.e.).
\end{enumerate}
\end{teo}
\begin{proof}
The existence of a unique  isomorphism $\Phi: \cal L^0(\Om_1)\to \cal L^0(\Om_2)$ with property (1)
follows directly from Theorem \ref{t4}.

(2) It is clear that $\Phi$ is the bijection from the algebra $\mathbb R(\nabla_{\mu_1})$ of step-elements
in $\cal L^0(\Om_1)$ onto the algebra $\mathbb R(\nabla_{\mu_2})$. By (\ref{e2}),
for every $f=\sum\limits_{i=1}^k \lambda_i e_i \in \mathbb R(\nabla_{\mu_1})$, we have
$$
\|\Phi(f)\|_1=\left \|\sum\limits_{i=1}^k \lambda_i \varphi(e_i)\right \|_1 =
\sum\limits_{i=1}^k |\lambda_i| \, \widehat\mu_2(\varphi(e_i))=
\sum\limits_{i=1}^k |\lambda_i| \, \widehat\mu_1(e_i) = \|f\|_1.
$$
Since the algebra $\mathbb R(\nabla_{\mu_i})$ is dense in  $(\cal L^1(\mu_i), \|\cdot\|_1), i=1, 2$, it follows that there exists a bijective linear isometry $U$ from $(\cal L^1(\Om_1), \|\cdot\|_1)$ onto $(\cal L^1(\Om_2), \|\cdot\|_1)$ such that $U(f) = \Phi(f)$ for every $f \in \mathbb R(\nabla_{\mu_1})$.
If $f \in \cal L^1_+(\Om_1)$, then there exists a sequence $\{f_n\} \su \mathbb R(\nabla_{\mu_1})$
such that $0 \leq f_n \uparrow f$.  Then we have  $\|f_n- f\|_1 \to 0$, which implies that $\|U(f_n)- U(f)\|_1 \to 0$. Therefore $\Phi(f_n)=U(f_n)\uparrow U(f)$. Since $\Phi: \cal L^0(\Om_1)\to \cal L^0(\Om_2)$ is an isomorphism, it follows that $\Phi(f_n)\uparrow \Phi(f)$, and we conclude that $\Phi(f) = U(f)$ for every $f \in \cal L^1_+(\Om_1)$.
Now, since $ \cal L^1(\Om_1)= \cal L_+^1(\Om_1)-\cal L^1_+(\Om_1)$, we obtain
$\Phi(x) = U(x)$ for all $\in \cal L^1(\Om_1)$,
hence $\Phi: (\cal L^1(\Om_1), \|\cdot\|_1) \to (\cal L^1(\Om_2), \|\cdot\|_1)$
 is a bijective linear isometry.

If $0 \leq f \in \cal L^{\ii}_+(\Om_1)$, we choose a sequence $\{f_n\} \subset \mathbb R(\nabla_{\mu_1})$ such that
$0 \leq f_n \uparrow f$ and $\|f_n- f\|_{\infty} \to 0$ and repeat the preceding argument.

(3) Since $\Phi: \cal L^0(\Om_1)\to \cal L^0(\Om_2)$ is an isomorphism, it follows that, given $f_n, f  \in  \cal L^0(\Om_1)$,
$$
\Phi(f_n) \stackrel{(o)}{\longrightarrow}  \Phi(f) \ \Longleftrightarrow \ f_n  \stackrel{(o)}{\longrightarrow}  f,
$$
where $ f_n  \stackrel{(o)}{\longrightarrow}  f$ is the $(o)-$convergence, that is, there exist $g_n, h_n \in \cal L^0(\Om_1)$
such that $g_n \leq f_n \leq h_n$ for all $n$ and $g_n \uparrow f, \ h_n \downarrow f$.
This completes the argument because $(o)-$convergence and convergence a.e. coincide
(see, for example, \cite[Ch.VI, \S 3]{vu}).
\end{proof}

We will also need some properties of conditional expectations.  

In the case of finite measure the following is known (see, for example, \cite[Ch.IV, \S IV.3]{ne}).

\begin{teo}\label{t3_5}Let $(\Om, \mathcal A, \mu)$ be a finite measure space, and let $\mathcal B$ be a
$\sigma-$subalgebra of $\mathcal A$. Then there exists a  positive linear contraction
$U:\cal L^1(\Om, \cal A, \mu)\rightarrow \cal L^1(\Om, \cal B, \mu)$ such that
\begin{enumerate}
\item $U(\cal L^1(\Om, \cal A, \mu)) =\cal L^1(\Om, \cal B, \mu)$ and
$U(\cal L^{\ii}(\Om, \cal A, \mu)) = \cal L^{\ii}(\Om, \cal B, \mu)$;

\item $U(\mathbf 1) = \mathbf 1$ and $\|U(f)\|_{\ii} \leq \|f\|_{\ii}$   for all $f \in \cal L^{\ii}(\Om, \mathcal A, \mu)$;

\item $U^2 = U$  and $\int_{\Omega} U(f) d \mu = \int_{\Omega} f d \mu$ for all $f \in \cal L^1(\Om, \mathcal A, \mu)$.
\end{enumerate}
\end{teo}

The next theorem is a version of Theorem \ref{t3_5} for $\sigma-$finite measure.

\begin{teo}\label{t3_6} Let $(\Om, \mathcal A, \mu)$ be a $\sigma-$finite measure space, and let $\cal B$ be a
$\sigma-$subalgebra of $\cal A$. Then there exists $S\in DS^+(\Om, \cal A, \mu)$ such that
\begin{enumerate}
\item $S(\cal L^1(\Om, \mathcal A, \mu)) =\cal L^1(\Om, \mathcal B, \mu)$ and  $S(\cal L^{\ii}(\Om, \mathcal A, \mu)) = \cal L^{\ii}(\Om, \mathcal B, \mu)$;

\item $S(\mathbf 1) = \mathbf 1$;

\item $S^2 = S$  and $\int_{\Omega} S(f) d \mu = \int_{\Omega} f d \mu$ for all $f \in \cal L^1(\Om, \mathcal A, \mu)$.
\end{enumerate}
\end{teo}

\begin{proof}
If $\mu(\Om) < \ii$, then Theorem \ref{t3_6} follows from Theorem \ref{t3_5}. So, let $\mu(\Om) = \ii$. Then there exists the sequence $\{G_n\}_{n=1}^\infty \subset \cal B$ such that $\mu(G_n) < \ii$ for all $n$, $G_n\cap G_k = \varnothing$
if $n\neq k$, and $\bigcup_{n=1}^\infty G_n = \Omega$. Let  $\cal A_{G_n} = \{G_n\cap E: E \in \cal A\}$ and
$\cal B_{G_n} = \{G_n\cap E: E \in \cal B\}, \ n \in \mathbb N$. By Theorem \ref{t3_5}, there exists a conditional expectation
$U_n:\cal L^1(\Om, \mathcal  A_{G_n}, \mu)\rightarrow \cal L^1(\Om, \mathcal B_{G_n}, \mu)$ satisfying
conditions (i)-(iii) of Theorem \ref{t3_5}.

Since $\int_{\Omega} |f| d \mu = \sum_{n=1}^\infty \int_{G_n} |f| d \mu$, $f \in \cal L^1(\Om, \mathcal A, \mu)$,
the map $U:\cal L^1(\Om, \mathcal A, \mu)\rightarrow \cal L^1(\Om, \mathcal B, \mu)$ given by
$$
U(f) =  \sum_{n=1}^\infty U_n( f\chi_{G_n})
$$
is well-defined.
It clear that $U$ is a positive linear $\|\cdot\|_1-$contraction such that
$\|U(f)\|_{\infty} \leq \|f\|_{\infty}$ for all $f \in \cal L^1(\Om, \mathcal A, \mu)\cap \cal L^{\infty}(\Om, \mathcal A, \mu)$
and $U(g) =g$ for all $g \in  L^1(\Om, \mathcal B, \mu)$.

By \cite[Theorem 3.1]{ccl}, there is $S \in DS^+(\Om, \cal A, \mu)$ such that $S(f) = U(f)$ for all
$ f \in \cal L^1(\Om, \cal A, \mu)$; in particular, $S(g) = g$ for all $g \in  L^1(\Om, \mathcal B, \mu)$.

If $0\leq f \in \cal L^{\ii}(\Om, \mathcal A, \mu)$, then $0\leq f\chi_{G_n} \in \cal L^{\ii}(\Om, \mathcal  A_{G_n}, \mu)$ and $\sum_{n=1}^k f\chi_{G_n} \uparrow f$ as $k \rightarrow \ii$. Since, by \cite[Theorem 3.1]{ccl},
$S| \cal L^{\ii}(\Om, \mathcal A, \mu)$ is $\sigma(\cal L^{\ii}, \cal L^1)-$continuous, it follows that
$$
\sum_{n=1}^k U_n( f\chi_{G_n})=S\left (\sum_{n=1}^k f\chi_{G_n}\right ) \uparrow S(f).
$$
as $k\to \ii$. Besides, for any $k$ we have
$$
\left \|U\left (\sum_{n=1}^k f\chi_{G_n}\right )\right \|_\ii \leq \left \|\sum_{n=1}^k f\chi_{G_n}\right \|_\ii \leq \| f\|_\ii,
$$
hence $S(\cal L^{\ii}(\Om, \cal A, \mu)) = \cal L^{\ii}(\Om, \mathcal B, \mu)$ and $\|S(f)\|_{\ii} \leq \|f\|_{\ii}$
if $f \in \cal L^{\ii}(\Om, \cal A, \mu)$.

From the definition of $U$ and $\int_{\Omega} |f| d \mu = \sum_{n=1}^\infty \int_{G_n} |f| d \mu$,
$f \in \cal L^1(\Om, \cal A, \mu)$, it follows that $S(\cal L^1(\Om, \mathcal A, \mu))=\cal L^1(\Om, \mathcal B, \mu)$.
It is also clear that  $S(\mathbf 1) = \mathbf 1$,  $S^2 = S$,  and $\int_{\Omega} S(f) d \mu = \int_{\Omega} f d \mu$ for all $f \in \cal L^1(\Om, \mathcal A, \mu)$.
\end{proof}

If $(\Om,\cal A,\mu)$ be a $\sigma-$finite measure space,
$E \in \mathcal A$ and $\cal A_{E} = \{E\cap G: G \in \cal A\}$, then $(E, \mathcal A_E, \mu)$ is also  a $\sigma-$finite measure space. We will need the next corollary of Theorem \ref{t3_6}.

\begin{cor}\label{t3_7} Let $(\Om, \mathcal A, \mu)$ be a $\sigma-$finite measure space,
and let  $E \in \cal A$. If $\cal B$  is a $\sigma-$subalgebra of
$\mathcal A_E$ and $T\in DS(E, \cal B, \mu)$, then there exists  $\widehat{T}\in DS(\Om, \cal A, \mu)$ such that
$$
\widehat T(g) = T(g) \ \ \text{and} \ \ A_n(\widehat T,g) =  A_n(T,g)
$$
for all $g \in \cal L^1(E, \cal B,\mu) + \cal L^{\ii}(E, \cal B,\mu)$ and $n$.
\end{cor}
\begin{proof}
By Theorem \ref{t3_6}, there  exists $S\in DS^+(E, \cal A_E, \mu)$ such that
$S(\cal L^1(E, \cal A_E, \mu)) =\cal L^1(E, \cal B, \mu)$, $S(\cal L^\ii(E, \cal A_E, \mu)) = \cal L^\ii(E, \cal B, \mu)$, and $S^2=S$.

If we define $\widehat{T}(f) = T(S(f\chi_E))$, $f \in \cal L^1(\Om, \mathcal A, \mu) + \cal L^{\ii}(\Om, \mathcal A, \mu)$,
then it follows that  $\widehat{T} \in DS(\Om, \cal A, \mu)$,  and $\widehat{T}(g) = T(g)$ and $A_n(\widehat{T},g) =  A_n(T,g)$
for all $g \in \cal L^1(E, \cal B,\mu) + \cal L^{\ii}(E, \cal B,\mu)$ and $n$.
\end{proof}

\section{Maximality of the space $\mc R_\mu$}

\begin{teo}\label{t2}
Let $\nu$ be Lebesgue measure on the interval $(0,\ii)$. Then, given $f\in \left (\mc L^1(0,\ii)+\mc L^\ii(0,\ii)\right ) \sm \mc R_\nu$, there exists $T\in DS$ such that the averages (\ref{e1}) do not converge a.e., hence $\mc C_\nu=\mc R_\nu$.
\end{teo}

\begin{proof}
Let $f\in (\mc L^1+\mc L^\ii)\sm \mc R_\nu$. If $f=f_+-f_-$, then either
$f_+\in (\mc L^1+\mc L^\ii)\sm \mc R_\nu$ or $f_-\in (\mc L^1+\mc L^\ii)\sm \mc R_\nu$,
so let us assume the former.  Then we have
$$\lim_{t\to \ii}\mu_t(f_+)>0,
$$
which implies that there exist $0<a<b<\ii$ such that
$$
\nu\{a\leq \mu_t(f_+)\leq b\}=\ii.
$$
If we denote $G=\{a\leq f_+\leq b\}$, then $\mu(G)=\ii$.
Since $\mu$ is non-atomic, one can construct a sequence $\{G_m\}_{m=0}^\ii$ of pairwise disjoint subsets of $G$ such that
$$
G=\bigcup_{m=0}^\infty G_m \text{ \ \ and \ \ } \mu(G_m)=1 \text{ \ for every \ } m.
$$
Let $0=n_0,n_1, n_2 \dots$ be an increasing sequence of integers. Define the function $\varphi$ on $(0,\ii)$ by
$$
\varphi(t)=\sum_{k=0}^\ii\left(\chi_{\bigcup_{m=n_k}^{n_{k+1}-1}G_m}(t)-\chi_{G_{n_{k+1}}}(t)\right).
$$
Note that $\varphi(t)=0$ when $t\in (0,\ii)\sm G$.

Next, since $\mu(G_m)=\mu(G_{m+1})$,  there exists a measure preserving isomorphism $\tau_m: G_m\to G_{m+1}$,
$m=0,1,2,\dots$ (\cite[Ch.2, Proposition 7.4]{bs}). Therefore, we have a measure preserving isomorphism $\tau: G\to G$ such that $\tau(G_m)=G_{m+1}$ for all $m$. Expand $\tau$ to $(0,\ii)$ by letting $\tau(t)=t$ if $t\in (0,\ii)\sm G$.

Let us now define $T\in DS(0,\ii)$ as
$$
T(g)(t)=\varphi(t)g(\tau(t)), \ \ t\in (0,\ii), \ \ g\in \mc L^1+\mc L^\ii.
$$
Then it follows that
$$
A_n(f_+)(t)=\frac 1n\left ( f_+(t)+\sum_{k=1}^{n-1}\varphi(t)\varphi(\tau(t))\cdot \dotsc \cdot \varphi(\tau^{k-1}(t))
f_+(\tau^k(t))\right ).
$$
We have
$$
A_{n_1}(f_+)(t)=\frac 1{n_1}\sum_{k=0}^{n_1-1}f_+(\tau^k(t))\ge a>\frac a2 \text{ \ \ for every \ } t\in G_0.
$$
Further, since
$$
A_{n_2}(f_+)(t)=\frac 1{n_2}\left ( \sum_{k=0}^{n_1-1}f_+(\tau^k(t))-\sum_{k=n_1}^{n_2-1}f_+(\tau^k(t)) \right )
\leq \frac 1{n_2} ( bn_1-a(n_2-n_1)),
$$
there exists such $n_2$ that
$$
A_{n_2}(f_+)(t)< -\frac a2 \text{ \ \ for every \ } t\in G_0.
$$
As
$$
A_{n_3}(f_+)(t)=\frac 1{n_3}\left ( \sum_{k=0}^{n_1-1}f_+(\tau^k(t))-\sum_{k=n_1}^{n_2-1}f_+(\tau^k(t))+
\sum_{k=n_2}^{n_3-1}f_+(\tau^k(t)) \right ),
$$
one can find $n_3$ for which
$$
A_{n_3}(f_+)(t)> \frac a2, \ \ t\in G_0.
$$
Continuing this procedure, we choose $n_1<n_2<n_3<\dots$ to satisfy the inequalities
$$
A_{n_{2k-1}}(f_+)(t)>\frac a2 \text{ \ \ and \ \ } A_{n_{2k}}(f_+)(t)<-\frac a2, \ \ k=1,2,\dots,
$$
for every $t\in G_0$, implying that the sequence $\{ A_n(f_+)(t) \}$ diverges whenever $t\in G_0$.

Finally, note that $T(f_-)=0$, which implies that
$$
A_n(f)=A_n(f_+)+f_-,
$$
hence the sequence $\{ A_n(f)(t) \}$ does not converge almost everywhere on $(0,\ii)$.
\end{proof}

Now we shall extend Theorem \ref{t2} to the class of $\sigma-$finite non-atomic measure spaces:

\begin{teo}\label{t8}
If $(\Om, \cal A, \mu)$ is a non-atomic $\sigma-$finite measure space, then $\cal C_\mu=\cal R_\mu$.
\end{teo}
\begin{proof}
Take  $f \in (\cal L^1(\Om) + \cal L^{\ii}(\Om))\sm \mc R_\mu$. We need to show that $f\notin \cal C_\mu$.
Without loss of generality, $f\ge 0$ (see the proof of Theorem \ref{t2}).
Assume that $\displaystyle \lim_{t \to \ii}\mu_t(f) =1$:
if $\displaystyle \lim_{t \to \ii}\mu_t(f) =\alpha >0$, then we can take $\frac{f}{\alpha}$ instead of $f$. Then we have
\begin{equation}\label{e3}
\mu\{f \geq 1\} = \nu\{\mu_t(f) \geq 1\}= \ii.
\end{equation}

Assume first that $\mu(E) = \infty$, where $E= \{f > 1\}$. In this case the functions $f$ and $h = f\chi_E$
are equimeasurable, so $\mu_t(f)=\mu_t(h)$. If $g=h-\chi_E$, then $\supp(g)=E$ and, given $t>0$,
$$
\mu_t(g) = \inf \{\lb>0: \ \mu\{g > \lb\} \leq t\}=\inf \{\lb>0: \ \mu\{h >1+ \lb\} \leq t\}=
$$
$$
=\inf \{\gamma -1>0: \ \mu\{h >\gamma\} \leq t\}= \inf \{\gamma>0: \ \mu\{h >\gamma\} \leq t\}-1=\mu_t(h)-1,
$$
hence $\displaystyle \lim_{t \rightarrow \infty}\mu_t(g) =0$.

By Theorem \ref{t7}, there exists a surjective m.p.t.
$$
\sigma:\supp(g)=E\longrightarrow \supp(\mu_t(g)) = (0,\ii)
$$
such that $g =  \mu_t(g) \circ \sigma$ on $E$.
Then we have
$$
h = g+\mathbf 1 =  (\mu_t(g)+1)\circ \sigma \text{ \ on \ } E.
$$
Since $\mu_t(g)+1  \in (\mc L^1(0,\ii)+\mc L^\ii(0,\ii)) \sm \mc R_\nu$, Theorem \ref{t2} entails
that there is $T\in DS(0,\ii)$ such that  the sequence
\begin{equation}\label{e4}
 \frac 1 n \sum_{k=0}^{n-1}T^k(\mu_t(g) + 1) \text{\ \ does not converge a.e. on\ \ } (0,\ii).
\end{equation}

If $\cal B$ is the $\sigma-$algebra of Lebesgue measurable sets in $(0,\ii)$, let us denote
$$
\mathcal A_\sigma= \{\sigma^{-1}(B): B \in \mathcal B\}, \ \ \nabla_\sigma = \{[G]: G \in \mathcal A_\sigma\}
$$
and let $\varphi([\sigma^{-1}(B)]) = [B]$, $B\in \cal B$. It is clear that $\varphi: \nabla_\sigma\to\nabla_\nu(0, \ii)$
is an  isomorphism of Boolean algebras and $\widehat{\nu}(\varphi(p)) = \widehat{\mu}(p)$ for all $p \in \nabla_\sigma$.
Besides, if  $B\in \cal B$, $e= \chi_B$, and
$e\circ \sigma = \chi_{\sigma^{-1}(B)}$, then $\chi_B\circ \sigma = \chi_{\sigma^{-1}(B)}$,  hence
$$
\varphi(e\circ \sigma) = \varphi(\chi_B\circ \sigma) = \varphi([\sigma^{-1}(B)])=[B] =e \text{ \ \ for all \ } e \in \nabla_\nu(0,\infty).
$$

By Theorem \ref{t5}, there exists an  isomorphism $\Phi: \cal L^0(E,\cal A_\sigma,\mu)\rightarrow \cal  L^0(0,\ii)$ such that
$\Phi(e)=\varphi(e)$ for all $e \in \nabla_\sigma$. Show that
\begin{equation}\label{e6}
\Phi(u \circ \sigma)=u \text{ \ \ for each \ } u \in   \cal  L^0(0,\ii).
\end{equation}
If $u=\sum\limits_{i=1}^k \lambda_i e_i \in \mathbb R(\nabla_\nu(0,\infty))$, where $\lb_i \in \mathbb R$ and
$e_i \in \nabla_\nu(0,\infty)$ with $e_i e_j = 0$ for $i \neq j$, then
$$
\Phi(u \circ \sigma)=\Phi\left (\left (\sum\limits_{i=1}^k \lambda_i e_i\right ) \circ \sigma\right ) =
\sum\limits_{i=1}^k \lambda_i \Phi(e_i \circ \sigma)= \sum\limits_{i=1}^k \lambda_i \varphi(e_i \circ \sigma)
= \sum\limits_{i=1}^k \lambda_i e_i =u.
$$
Let now $u \in   \cal  L_+^0(0,\ii)$. Then there exists a sequence $0\leq u_n \in \mathbb R(\nabla_\nu(0,\infty))$ such that $u_n \uparrow u$, hence $(u_n \circ \sigma) \uparrow (u \circ \sigma)$.  Since $\Phi$ is an isomorphism, it follows that
$\Phi(u_n \circ \sigma) \uparrow \Phi(u \circ \sigma)$. Therefore
$$
 \Phi(u \circ \sigma) = \sup_n \Phi(u_n \circ \sigma) = \sup_n u_n =u.
$$
If $u \in  \cal L^0(0,\ii)$, then $u=u_{+}-u_{-}$ with $u_+, u_- \in  \cal  L_+^0(0,\ii)$, and (\ref{e6}) follows.

Let $\widehat{T} = \Phi^{-1}T\Phi$. In view of Theorem \ref{t5}\,(2), $\widehat{T}\in DS(E,\cal A_\sigma, \mu)$.
Note that, since $\sigma$ is a m.p.t. and $g =  \mu_t(g) \circ \sigma$, it follows that $f, g, h \in \cal L^0(E, \cal A_\sigma, \mu)$.
Then, utilizing (\ref{e6}), we obtain
$$
\frac 1 n \sum_{k=0}^{n-1}\widehat{T}^k(h) = \frac 1 n \sum_{k=0}^{n-1}(\Phi^{-1}T\Phi)^k (h)=\Phi^{-1}\frac 1 n \sum_{k=0}^{n-1}T^k(\Phi(h))=
$$
$$
=\Phi^{-1}\frac 1 n \sum_{k=0}^{n-1}T^k(\mu_t(g) +1 ).
$$
Combining now (\ref{e4}) with  Theorem \ref{t5}\,(3), we conclude that the averages $A_n(\widehat T,f)$
do not converge a.e. on  $E$.

By Corollary \ref{t3_7}, there exists $\widetilde T\in DS(\Om, \cal A, \mu)$
such that $\widetilde T(u) = \widehat{T}(u)$ and $A_n(\widetilde T,u) =  A_n(\widehat{T},u)$
for all $u \in \cal L^0(E,\cal A_\sigma, \mu)$.
Therefore, $A_n(\widetilde T,f)$ do not converge a.e. on  $E$, hence $f\notin \cal C_\mu$.

If $\mu(E) < \infty$, let $E_1 = \{f = 1\}$. Then, in view of (\ref{e3}), we have
$$
\mu(E_1) = \nu\{\mu_t(f) = 1\} = \infty.
$$
In particular, if $g_1= f\cdot \chi_{E_1} -\chi_{E_1} = 0$,  Theorem \ref{t7} implies that
there exists a surjective m.p.t. $\sigma_1:\Om =\supp(g_1)\to(0,\ii)$.
Repeating the argument as in the case $\mu(E) = \infty$, we complete the proof.
\end{proof}

Next we need a version of Theorem \ref{t2} for a totally atomic infinite measure space
with the atoms of equal measures.
If  $(\Om, \mu)$ is  such a measure space, then $\Om=\{\omega_n\}_{n=1}^\ii$,
where $\omega_n$ is an atom with $0<\mu(\omega_n) = \mu(\omega_{n+1}) <\ii$ for all $n$.
In this case, the algebra $\cal L^0(\Om)$ is isomorphic to the algebra of sequences of real numbers,
the  algebra $\cal L^{\infty}(\Om)$ is isomorphic to the algebra
$$
l^\ii=\left \{f=\{\alpha_n\}_{n=1}^\ii:\ \alpha_n \in\mathbb{R},\ \|f \|_\ii
=\sup_n|\al_n|<\ii\right \},
$$
and  the algebra $\cal L^1(\Om)$ is isomorphic to the  algebra
$$
l^1=\left \{f=\{\alpha_n\}_{n=1}^\ii:\ \alpha_n \in\mathbb{R},\ \|f \|_1
=\sum_{n=1}^\ii |\al_n|<\ii\right \}.
$$
As $l^1\su l^\ii$, we have $l^1+ l^\ii = l^\ii$.  Note also that if
$f_k =\{\alpha^{(k)}_n\}_{n=1}^{\infty}\in l^\ii$ and $f=\{\alpha_n\}_{n=1}^{\infty} \in l^\ii$, then
$$
f_k \rightarrow f \ \  \mu-a.e. \ \Longleftrightarrow \ \alpha_n^{(k)} \rightarrow \alpha_n \text{ \ as \ } k\to \ii
\text{ \ for every \ }n.
$$

\begin{teo}\label{t9}
If  $(\Om,\mu)$ is a totally atomic infinite measure space with the atoms of equal measures, then $\cal C_\mu=\cal R_\mu$.
\end{teo}
\begin{proof}
We need to show that if  $f =\{\alpha_n\}_{n=1}^{\infty} \in l^\ii \sm \mc R_\mu$, then $f\notin \cal C_\mu$.
As in the proof of Theorem \ref{t8}, we may assume that
$\alpha_n \geq 0$ for every $n$, and $\displaystyle \lim_{t \to \ii}\mu_t(f) =1$. Identifying $\Om=\{\omega_n\}_{n=1}^\ii$
with $\mathbb N$, it is clear that
$$
\mu\{n: \alpha_n \geq 1\} = \nu\{\mu_t(f) \geq 1\} = \ii.
$$
Besides, there exists $\alpha > 1$ such that the set
$$
G = \{m \in \Om : 1 \leq \alpha_m \leq \alpha\}= \{m_1 < m_2<\dotsc < m_s<\dots \}
$$
is infinite.

Let $1=n_0,n_1, n_2 \dots$ be an increasing sequence of integers. Define the function $\varphi:\Om \to \mathbb R$  
$(E, \mathcal A_E, \mu)$ by
$$
\varphi(m)=\left(\chi_{\{m_1, m_{2}, ..., m_{n_{1}}-1\}}(m)-
\chi_{\{m_{n_{1}}\}}(m)\right) +
$$
$$
+\sum_{k=1}^\ii\left(\chi_{\{m_{n_k},m_{n_k+1}, ...,m_{n_{k+1}}-1\}}(m)-
\chi_{\{m_{n_{k+1}}\}}(m)\right) \text{ \ if \ } m\in G;
$$
$$
\varphi(m)=0 \text{ \ if \ } m\notin G.
$$
Let $\tau: \Om \to \Om$ be defined as
$$
\tau(m_i) = m_{i+1} \text{ \ if\ } m_i \in G \text{ \ and \ }\tau(m) = m \text{ \ if \ } m \notin G.
$$
Finally, define  the linear operator $T:l^\ii\rightarrow l^\ii$ by
$$
T(g)(n)=\varphi(n)g(\tau(n)), \ \ g \in l^\ii.
$$
It is clear that $T(l^1) \subset l^1$, $\|T\|_{l^1 \rightarrow l^1} \leq 1$, and $\|T\|_{l^\ii\rightarrow l^\ii} \leq 1$,
that is, $T\in DS(\Om)$.

We have
$$
A_{n_1}(T,f)(m_1)=A_{n_1}(f)(m_1)=\frac 1{n_1}\sum_{k=0}^{n_1-1}f(\tau^k(m_1))\ge 1.
$$
Further, since
$$
A_{n_2}(f)(m_1)=\frac 1{n_2}\left ( \sum_{k=0}^{n_1-1}f(\tau^k(m_1))-\sum_{k=n_1}^{n_2-1}f(\tau^k(m_1)) \right )
\leq \frac 1{n_2} ( \alpha n_1-(n_2-n_1)),
$$
there exists such $n_2$ that
$$
A_{n_2}(f)(m_1)< -\frac 12.
$$
Repeating remaining steps of the proof of Theorem \ref{t2}, we see that the sequence $\{ A_n(T,f)(m_1)\}$
does not converge, that is, $f\notin \cal C_\mu$.
\end{proof}

Finally, we establish Theorem \ref{t2} for  a quasi-non-atomic measure space:

\begin{teo}\label{t10}  If an infinite measure space $(\Om, \cal A, \mu)$  is quasi-non-atomic,
then $\cal C_\mu=\cal R_\mu$.
\end{teo}
\begin{proof}
 We need to show that if $f\in (\cal L^1+\cal L^\ii)\sm \cal R_\mu$, then
$f\notin \cal C_\mu$.
Consider the complete Boolean algebra $\nabla_\mu = \{[E]: E \in \mathcal A\}$.
There exists $e \in \nabla_\mu$ such that $e \nabla_\mu$ is a non-atomic and
$(\mathbf 1-e) \nabla_\mu$ is a totally atomic Boolean algebra.
Since  $\displaystyle \lim_{t \rightarrow \infty}\mu_t(f) >0$ and $f = ef + (\mathbf 1-e)f$,
it follows from the inequality
$$
\mu_{t+s}(ef + (\mathbf 1-e)f) \leq \mu_t(ef) + \mu_s((\mathbf 1-e)f)
$$
(see \cite[Ch.2, Proposition 1.3]{bs}) that
$$
\lim_{t \rightarrow \infty}\mu_t(ef) > 0 \text{ \ \ or \ } \lim_{t \rightarrow \infty}\mu_t((\mathbf 1-e)f) > 0.
$$

Assume that the number of atoms in $(\Om,\mu)$ is finite.
Then $\displaystyle \lim_{t \rightarrow \infty}\mu_t((\mathbf 1-e)f)=0$, implying that
$\displaystyle \lim_{t \rightarrow \infty}\mu_t(ef) > 0$, hence
\begin{equation}\label{e5}
ef \in \left (\cal L^1(e \nabla_\mu) +
\cal L^{\ii}(e \nabla_\mu)\right )\sm \mc R_\mu.
\end{equation}
By Theorem \ref{t8}, there exists $T\in DS(e\nabla_\mu)$ such that the averages $A_n(T,ef)$ do not converge a.e.

Besides, Corollary \ref{t3_7} implies that there is $\widehat T\in DS(\Om,\mu)$ such that
$$
\widehat T(h)=T(h) \text{ \ \ and \ \ } A_n(\widehat T,h)=A_n(T,h), \ \ h\in \cal L^1(e\nabla_\mu)+\cal L^\ii(e\nabla_\mu),
\ n\in \mathbb N.
$$
Next, defining the operator $\widetilde T\in DS(\Om,\mu)$ by
$$
\widetilde T(g) = \widehat T(eg), \ \ g \in \cal L^1(\Om,\mu) + \cal L^{\ii}(\Om,\mu),
$$
we obtain
$$
A_n(\widetilde T,f) = A_n(\widehat T, ef)=A_n(T,ef),\ \ n\in \mathbb N,
$$
hence $f\notin \cal C_\mu$.

Let now $(\Om,\mu)$ have infinitely many items of the same measure.
Then there exists
$e \in \nabla_\mu$ such that either $ e\nabla_\mu  \neq \{0\}$ is a non-atomic Boolean algebra and (\ref{e5}) holds or
$(\mathbf 1-e) \nabla_\mu\neq \{0\}$ is a totally atomic Boolean algebra with the atoms of equal measures and, with
$e^\perp=\mathbf 1-e$,
$$
e^\perp f \in \left (\cal L^1(e^\perp \nabla_\mu) + \cal L^{\ii}(e^\perp \nabla_\mu) \right )\sm \mc R_\mu.
$$

In the first case,  we find $T$ and define $\widetilde T\in DS(\Om,\mu)$ as above and conclude that the averages
$A_n(\widetilde T, f)$ do not converge $\mu-$a.e., that is, $f\notin \cal C_\mu$.

In the second case, we use Theorem \ref{t9} and repeat the preceding argument.
\end{proof}

\section{Applications}
Combining Proposition \ref{p1} with Theorems \ref{t1} and \ref{t10}, we arrive at the following
characterization of the fully symmetric spaces possessing individual ergodic theorem property.
\begin{teo}\label{t210}
Let  an infinite measure space $(\Om,\mu)$   be quasi-non-atomic.
Given a fully symmetric space $E \subset \cal L^1(\Om)+\cal L^\ii(\Om)$, the following conditions are equivalent:
\begin{enumerate}
\item $E\in IET(\Om)$;
\item $E\su \cal R_\mu$;
\item $\mathbf 1\notin E$.
\end{enumerate}
\end{teo}

To outline the scope of applications of Theorem \ref{t210}, we present some examples of
fully symmetric spaces $E$ such that $\mathbf 1\notin E$ or $\mathbf 1\in E$.

1. Let $\Phi$ be an {\it Orlicz function}, that is, $\Phi:[0,\ii)\to [0,\ii)$ is left-continuous, convex,  increasing function such that $\Phi(0)=0$ and $\Phi(u)>0$ for some $u\ne 0$ (see, for example \cite[Chapter 2, \S 2.1]{es}).  Let
$$
\cal L^\Phi=\cal L^\Phi(\Om,\mu)=\left \{ f \in \cal L^0(\Om,\mu): \  \int_{\Om} \left (\Phi\left (\frac {|f|}a \right )\right ) d \mu
<\ii \text { \ for some \ } a>0 \right \}
$$
be the corresponding {\it Orlicz space},  and let
$$
\| f\|_\Phi=\inf \left \{ a>0:  \int_{\Om} \left (\Phi\left (\frac {|f|}a \right )\right ) d \mu \leq 1 \right \}
$$
be the {\it Luxemburg norm} in $\cal L^\Phi$. \ It is well-known that  $(\cal L^\Phi, \| \cdot\|_\Phi)$
is a fully symmetric space.

Since  $\mu(\Om) = \infty$, if $\Phi(u)>0$ for all $u\ne 0$, then
$\int_{\Om} \left (\Phi\left (\frac {\mathbf 1}a \right )\right ) d \mu = \ii$ for each
$a>0$, hence $\mathbf 1  \notin   \cal L^\Phi$. If  $\Phi(u)=0$ for all $0\leq u< u_0$, then  $\mathbf 1 \in \cal  L^\Phi$.
Therefore, Theorem \ref{t210} implies the following.

\begin{teo}\label{t211}
Let  $(\Om,\mu)$ be an infinite measure space, and let  $\Phi$ be an Orlicz function.
If $\Phi(u)>0$ for all $u\ne 0$, then $(\cal L^\Phi, \|\cdot\|_\Phi) \in IET(\Om)$. If $(\Om,\mu)$
is  quasi-non-atomic and $(\cal L^\Phi, \|\cdot\|_\Phi) \in IET(\Om)$,
then $\Phi(u)>0$ for all $u\ne 0$.
\end{teo}

2. A  symmetric space  $(E, \| \cdot \|_E)$ is said to have an {\it order continuous norm} if
$$\| f_{n}\|_E\downarrow 0 \ \ \text{whenever} \ \ f_{n}\in E_+ \ \ \text{and} \ \  f_{n}\downarrow 0.
$$
If $E$  is a symmetric space  with order continuous norm, then
$\mu\{|f| > \lambda\}< \infty$  for all $f \in E$ and $\lambda > 0$, so  $E \su \cal R_\mu $; in particular, $\mathbf 1  \notin   E$.

\vskip 7pt
3.  Let $\varphi$ be an increasing  concave function on $[0, \infty)$ with $\varphi(0) = 0$ and
$\varphi(t) > 0$ for some $t > 0$, and let
$$
\Lambda_\varphi=\Lambda_\varphi(\Om,\mu) = \left \{f \in \cal L^0(\Om,\mu): \ \|f \|_{\Lambda_\varphi} =
\int_0^{\infty} \mu_t(f) d \varphi(t) < \infty \right \},
$$
the corresponding {\it Lorentz space}.

It is well-known that $(\Lambda_\varphi, \|\cdot\|_{\Lambda_\varphi})$
is a fully symmetric space; in addition, if $\varphi(\infty) = \ii$, then
$\mathbf 1  \notin  \Lambda_\varphi$ and if $\varphi(\infty) < \ii$, then
$\mathbf 1  \in  \Lambda_\varphi$.
Another application of Theorem \ref{t210} yields the following.

\begin{teo}\label{t212}
Let  $(\Om,\mu)$ be an infinite measure space,  and let  $\varphi$ be an increasing  concave function on $[0, \infty)$ with $\varphi(0) = 0$ and
$\varphi(t) > 0$ for some $t > 0$. If $\varphi(\infty) = \ii$, then  $\Lambda_\varphi \in IET(\Om)$.
If $(\Om,\mu)$ is
quasi-non-atomic and $\Lambda_\varphi \in IET(\Om)$, then $\varphi(\infty) = \ii$.
\end{teo}

4. Let $E=E(0,\ii)$ be a fully symmetric space. If $s>0$, let the bounded linear operator
$D_s$ in $E$  be given by $D_s(f)(t) = f(t/s), \ t > 0$.
The  {\it Boyd index}  $q_E$ is defined as
$$
q_E=\lim\limits_{s \to +0}\frac{\log s}{\log \|D_{s}\|}.
$$
It is known that $1\leq q_E\leq \ii$  \cite[Vol.II, Ch.II, \S 2b,  Proposition 2.b.2]{lt}.
Since  $\|D_{s}\| \leq \max\{1,s\}$ \cite[Vol.II, Ch.II, \S 2b]{lt}, $\mathbf 1 \in E$ would imply
$D_{s}(\mathbf 1) =\mathbf 1$  and \ $\|D_{s}\| =1$ for all $s \in (0,1)$, hence $q_E = \infty$.
Thus, if  $q_E < \ii$, we have $\mathbf 1\notin E$.

\end{document}